\documentclass[12pt]{article}

\usepackage{amssymb,amsmath,xcolor} 
\usepackage{amsfonts}  
\usepackage{amsmath}  
\usepackage{amssymb}  
\usepackage{graphicx}  
\graphicspath{{quasi-inv5_graphics/}{quasi-inv5_tcache/}{quasi-inv5_gcache/}}
\DeclareGraphicsExtensions{.pdf,.eps,.ps,.png,.jpg,.jpeg}
\providecommand{\U}[1]{\protect\rule{.1in}{.1in}}
\newtheorem{theorem}{Theorem}

\newtheorem{corollary}[theorem]{Corollary}

\newtheorem{lemma}[theorem]{Lemma}

\newtheorem{proposition}[theorem]{Proposition}

\newenvironment{proof}[1][Proof]{\noindent\textbf{#1.} }{\ \rule{0.5em}{0.5em}}
\begin{document}
\title{Principal orbit type theorems for reductive algebraic group actions and the Kempf--Ness Theorem}
\author{Nolan R. Wallach}
\maketitle
\begin{abstract}The main result asserts: Let $G$ be a reductive, affine algebraic group and let $(\rho  ,V)$ be a regular representation of $G$. Let $X$ be an irreducible $\mathbb{C}^{ \times } G$ invariant Zariski closed subset such that $G$ has a closed orbit that has maximal dimension among all orbits (this is equivalent to: generic orbits are closed). Then there
exists an open subset, $W$,of $X$ in the metric topology which is dense with complement of measure $0$ such that if $x ,y \in W$ then $\left (\mathbb{C}^{ \times } G\right )_{x}$ is conjugate to $\left (\mathbb{C}^{ \times } G\right )_{y}$. Furthermore, if $G x$ is a closed orbit of maximal dimension and if $x$ is a smooth point of $X$ then there exists $y \in W$ such that $\left (\mathbb{C}^{ \times } G\right )_{x}$ contains a conjugate of $\left (\mathbb{C}^{ \times } G\right )_{y}$. The proof involves using the Kempf-Ness theorem to reduce the result to the principal orbit type theorem for compact Lie groups. 
\end{abstract}

\section{Introduction}
This paper is a spin-off from some joint work \cite{GoKrWa} and
\cite{SaWaGaKr} related to quantum information theory. In the first of these
papers we proved a result that could have been derived from the Luna {\'E}tal Slice Theorem \cite{Luna}.
But since physicists were the intended audience, we decided to give a proof of the needed geometric result that ``only''\ used
the principal orbit type theorem for compact transformation groups. Luna's theorem does not directly imply the result in either paper and in the second
paper the method in the first using invariant theory to reduce the needed result to one in which the slice theorem applies was unavailable. Thus the reduction
to compact groups was the simplest route in that case. 

In Richardson \cite{Richardson}
(Theorem 5.3), Luna \cite{Luna} (Theorem 8) there is a proof that if $X$ is a smooth affine $G$ space with $G$ reductive acting morphically then there is a Zariski open, non-empy subset $U \subset X$ such that if $x ,y \in U$ then $G_{x} =\{g \in G\vert g x =x\}$ and $G_{y}$ are conjugate. 

The main results in this paper prove similar theorems under stronger hypotheses on the group action.
The main result asserts: Let $G$ be a reductive, affine algebraic group and let $(\rho  ,V)$ be a regular representation of $G$. Let $X$ be an irreducible $\mathbb{C}^{ \times } G$ invariant Zariski closed subset such that $G$ has a closed orbit that has maximal dimension among all orbits (it his is equivalent to: generic orbits are closed). Then there
exists an open subset, $W$,of $X$ in the metric topology which is dense with complement of measure $0$ such that if $x ,y \in W$ then $\left (\mathbb{C}^{ \times } G\right )_{x}$ is conjugate to $\left (\mathbb{C}^{ \times } G\right )_{y}$. Furthermore, if $G x$ is a closed orbit of maximal dimension and if $x$ is a smooth point of $X$ then there exists $y \in W$ such that $\left (\mathbb{C}^{ \times } G\right )_{x}$ contains a conjugate of $\left (\mathbb{C}^{ \times } G\right )_{y}$. 

An example of the difference between the two theorems the latter implies directly that if $X$ is smooth (e.g. $X$ is affine space and $G$ acts linearly) and if there is one closed $G$--orbit, $G x$, such that $\left (\mathbb{C}^{ \times } G\right )_{x} =\{I\}$ then the generic isotropy group is trivial (this is the result proved in \cite{GoKrWa}
and \cite{SaWaGaKr}). The Luna {\'E}tal Slice theorem implies that the
generic isotropy group for $G$ is trivial but this does not directly imply the result for $\mathbb{C}^{ \times } G$. (See example 2 in the last section.) 

The Richardson-Luna theorem described above implies that the set
$W$ in our theorem contais a non--empty Zariski open subset. Notice, that in many examples only one orbit of $\mathbb{C}^{ \times } G$ is closed in $X$ so the Luna {\'E}tal Slice Theorem does not directly apply. 

We also note that the same arguments prove
the analogue of the above result for $G$ acting on an Zariski closed irreducible subst of $V$. This result is a direct consequence of the Luna {\'E}tal Slice Theorem and the Richardson-Luna principal orbit theorem applied
to the smooth points. We feel that there is still a benefit to having our argument in the literature since it is based on the the Kempf-Ness Theorem (an
amazing interaction between Freshman Calculus and the Hilbert-Mumford Theorem) and the Principal Orbit Type Theorem for compact Lie groups which are more
elementary than the Luna {\'E}tal Slice theorem. 

In the last section we give two examples. The first is an action on affine space
where there is an orbit with thrivial stabilizer but the generic orbit $G$ has stabilizer of order 8. This example also shows why one can't apply the theorem for compact Lie groups directly. Since in
this example the generic stabilizer for a maximal compact subgroup must be trivial. 

We thank Hanspeter Kraft for pointing out a blunder
in an early version of this note. We also thank our co-authors in \cite{SaWaGaKr}
for pushing for a path to the geometric results we needed that could be understood by non-experts in algebraic transformation groups. 

\section{The Kempf-Ness Theorem.}
Let $G$ be a Zariski closed, reductive algebraic subgroup of $G L (n ,\mathbb{C})\text{.}$ We can assume that $g^{ \ast } \in G$ for all $g \in G$ ($g^{ \ast }$ the conjugate transpose of $g$) see e.g. \cite{GIT}, Theorem 3.13. Let $K =G \cap U (n)$. Then $K$ is a maximal compact subgroup of $G$ and $G$ is the Zariski closure of $K$ in $G L (n ,\mathbb{C})$. Let $V$ be the $G$--module $\mathbb{C}^{n}$ with the $G$ action as a subgroup of $G L (n ,\mathbb{C})$. Let $ \langle \ldots  ,\ldots  \rangle $ equal to the standard Hermitian inner product
on $V$. We say that element $v \in V$ is critical if $f_{X} (v) = \langle X v ,v \rangle  =0 ,X \in L i e (G)\text{.}$ We use the notation $C r i t (V)$ for the set of critical elements. In the literature this set is usual called the Kempf-Ness set. Then $C r i t (V)$ is a real algebraic sub-variety of $V$. Since \begin{equation*}L i e (G) =L i e (K) \oplus i L i e (K)
\end{equation*} as a real subspace the $f_{X}$ with $X \in L i e (K)$ define $C r i t (V)\text{.}$ 

Recall the Kempf-Ness Theorem \cite{Kempf-Ness}
(cf.\cite{GIT} Theorem 3.26).

\begin{theorem}
\label{Kempf-Ness}Let $G ,K$ be as above. Let $v \in V$. 

1. $v$ is critical if and only if $\left \Vert g v\right \Vert  \geq \left \Vert v\right \Vert $ for all $g \in G$. 

2. If $v$ is critical and $X \in L i e (G)$ is such that $X^{ \ast } =X$ and if $\left \Vert e^{X} v\right \Vert  =\left \Vert v\right \Vert $ then $X v =0$. 

3. If $v ,w$ are critical and $w \in G v$ then $w \in K v$. 

4. If $G v$ is closed then $G v \cap C r i t (V) \neq \varnothing $. 

5. If $v$ is critical then $G v$ is closed. 
\end{theorem}

\begin{corollary}
If $v \in C r i t (V)$ then $G_{v} =\{g \in G\vert g v =v\}$ is invariant under adjoint relative to $ \langle \ldots  ,\ldots  \rangle $. Thus $K_{v} =K \cap G_{v}$ is maximal compact in $G_{v}$ and $G_{v}$ is the Zariski closure of $K_{v}$ in $G$. 
\end{corollary}

\begin{proof}
Let $g \in G_{v}$. $g =k e^{X}$ (cf. \cite{GIT} Theorem 2.16) with $X \in L i e (G)$ such that $X^{ \ast } =X$ and $k \in K$ then \begin{equation*}\left \Vert v =\right \Vert \left \Vert k e^{X} v\right \Vert  =\left \Vert e^{X} v\right \Vert \text{.}
\end{equation*} Thus $X v =0.$ Hence $k v =v$ so $g^{ \ast } =e^{X} k^{ -1} \in K v$. 
\end{proof}

\section{Some consequences}
Let $G$ be a reductive, affine algebraic group and let $X$ be an irreducible affine variety on which $G$ acts. It is standard that there exists a regular representation $(\sigma  ,V)$ of $G$, Zariski closed $\sigma  (G)$--invariant subset, $Y$, of $V$ and an equivariant isomorphism $\Psi  :X \rightarrow Y$. For simplicity we replace $G$ with $\sigma  (G)$ and $X$ with $Y$. We can so assume that there is an inner product $ \langle \ldots  ,\ldots  \rangle $ on $V$ such that $G$ is invariant under adjoint with respect to $ \langle \ldots  ,\ldots  \rangle $. (cf. \cite{GIT}
Proposition 3.3 and Theorem 3.13). Thus $G$, $V$ is exactly as in the previous section. Let $K$ and $C r i t (V)$ have the meanings above. 

We set $C r i t (X) =X \cap C r i t (V)$. If $x \in X$ then $G x$ is closed $X$ if and only if it is closed in $V$. Thus $X^{c}$, the union of the closed orbits in $X$, is equal to $X \cap V^{c}$. Also $X^{c} =G C r i t (X)$. Let \begin{equation*}d (X) =\max_{x \in X^{c}}G x ,d_{1} (X) =\max_{x \in X}G x\text{.}
\end{equation*} Set \begin{equation*}C r i t^{o} (X) =\{x \in C r i t (X)\vert \dim K x =d (X)\}\text{.}
\end{equation*} We also set $X_{r e g}$ \ equal to the set of smooth points of $X$ and $X_{d (X)}^{c}$ equal to the union of the closed $G$--orbits in $X$ of dimension equal to $d (X)$. 

In addition we set \begin{equation*}X_{ \leq r} =\{x \in X\vert \dim G x \leq r\}
\end{equation*} and\ \begin{equation*}X_{ \geq r} =\{x \in X\vert \dim G x \geq r\}
\end{equation*} then $X_{ \leq r}$ is Zariski closed in $X$, $X_{ \geq r}$ is Zariski open in X and both are $G$--invariant.

\begin{lemma}
Assume that $d_{1} (X) =d (X)$ then $X_{d (X)}^{c}$ is Zariski open in $X$. 
\end{lemma}

\begin{proof}
We may assume that $G$ is connected since $\dim G x =\dim  G^{o} x$ if $G^{o}$ is the identity component of $G$. Set $d =d (X)$. Let $x \in X_{d}^{c}$. Then $G x\;$is closed and $G x \cap X_{ \leq d -1} =\varnothing $. Thus there exists $u_{x} \in \mathcal{O} (X)^{G}$ such that $u_{x} (G x) =\{1\} ,u_{x} (X_{ \leq d -1}) =\{0\}$(cf. \cite{GIT} 3.13). We assert that the principal
open set \begin{equation*}X_{u_{x}} =\{y \in X\vert u_{x} (y) \neq 0\}
\end{equation*} is contained in $X_{d}^{c}$. Indeed, $X_{u_{x}} \cap X_{ \leq d -1} =\varnothing $. Hence, if $y \in X_{u_{x}}$ then $\dim G y =d$. Since $u_{x}$ is constant on the closure of $G y$, every $G$--orbit in the closure of $G y$ has dimension $d$. This implies that $G y$ is closed. Clearly \begin{equation*} \cup _{x \in X_{d}^{c}}X_{u_{x}} =X_{d}^{c}\text{.}
\end{equation*} 
\end{proof}

\begin{corollary}
\label{reg}If $d (x) =d_{1} (X)$ then $X_{d (X)}^{c} \cap X_{r e g}$ is Zariski open and non-empty. 
\end{corollary}

\begin{proposition}
\label{smooth}If $d (X) =d_{1} (X)$ then $C r i t^{o} (V) \cap X_{r e g}$ is a real submanifold of $X_{r e g}$ of dimension $2\dim X -d (X)\text{.}$ 
\end{proposition}

\begin{proof}
We note that, since $X$ is irreducible and both $X_{r e g}$ and $X_{d}^{c}$ are Zariski open and non--empty, hence $X_{r e g} \cap X_{d}^{c} \neq \varnothing $. 

Let $m =\dim K$. Let $v \in C r i t^{o} (X)$ then $\dim  K_{v} =m -d\text{.}$ Let $X_{1} ,\ldots  . , X_{m}$ be a basis of $L i e (K)$ such that $X_{d +1} ,\ldots  ,X_{m}$ is a basis of $L i e (K_{v})\text{.}$ Set \begin{equation*}U =\{u \in X\vert X_{1} u \wedge \cdots  \wedge X_{d} u \neq 0\}\text{.}
\end{equation*} Then $U$ is Zariski open in $X$ and $v \in U$. Set \begin{equation*}g_{j} (x) =\frac{1}{2 i} \langle X_{j} x ,x \rangle \text{.}
\end{equation*} Then $g_{j}$ is real valued. We assert that \begin{equation*}C r i t (V) \cap U =\{u \in U\vert g_{j} (u) =0 ,1 \leq j \leq d\}\text{.}
\end{equation*} Let $u \in C r i t (V) \cap U$. By the definition of $d$, if $m \geq j >d$ then \begin{equation*}X_{j} u =\sum _{k \leq d}a_{k ,j} (u) X_{k} u\text{.}
\end{equation*} So \begin{equation*}g_{j} (u) =\sum _{k \leq d}a_{k ,j} g_{k} (u) =0.
\end{equation*} Let \begin{equation*}\omega  (x ,y) =\mathrm{I} \mathrm{m} \langle x ,y \rangle 
\end{equation*} for $x ,y \in V$. Then $\omega $ is a nondegenerate, alternating $\mathbb{R}$--bilinear form on $V$ which is non-degenerate on any complex subspace of $V$ in particular on $T_{x} (X)$ for $x$ a smooth point of $X$. Furthermore, if $w \in V$ then\begin{equation*}\left (d g_{j}\right )_{v} (w) =\omega  (X_{j} v ,w)\text{,}
\end{equation*} Thus\begin{equation*}d g_{1 v\vert T_{v} (X)} ,\ldots  ,d g_{d v\vert T_{v} (X)}
\end{equation*} are linearly independent. This implies that $C r i t (X) \cap U \cap X_{r e g}$ is a (real) submanifold of $U$ of dimension $2\dim X -d$. \ Clearly, $C r i t (X) \cap U \cap X_{r e g} \subset C r i t^{o} (V) \cap X_{r e g}$ and $C r i t^{o} (V) \cap X_{r e g}$ is the union of these submanifolds. 
\end{proof}

\begin{theorem}
\label{main-point}Assume that $d =d_{1}$, $C r i t^{o} (X) \cap X_{r e g}$ is non-empty and connected. 
\end{theorem}

\begin{proof}
We can assume that $G$ is connected since $X_{d}^{c}$ and $C r i t^{o} (X)$ are the same for $G$ and its identity component. We consider the map \begin{equation*}\Psi  :G \times C r i t^{o} (X) \cap X_{r e g} \rightarrow X_{r e g} ,g ,x \longmapsto g x\text{.}
\end{equation*} We assert the the differential of $\Psi $ is surjective. Indeed, the tangent space at $x \in C r i t^{o} (X) \cap X_{r e g}$ (thought of as a subspace of $V$) is \begin{equation*}T_{x} (C r i t^{o} (X) \cap X_{r e g}) =\{w \in V\vert \omega  (L i e (K) x ,w) =0\}
\end{equation*} (see the proof of Proposition \ref{smooth}). If $X ,Y \in L i e (K)$, $x \in C r i t^{o} (X) \cap X_{r e g}$ and if $\omega $ is as above then \begin{equation*}\omega  (X x ,i Y x) =\mathrm{I} \mathrm{m} \langle X x ,i Y x \rangle  = -\mathrm{R} \mathrm{e} \langle X x ,Y x \rangle \text{.}
\end{equation*} We also note that $\mathrm{R} \mathrm{e} \langle \ldots  ,\ldots  \rangle $ defines a positive definite symmetric form on $V$ as a vector space over $\mathbb{R}$. Thus \begin{equation*}i L i e (K) x \cap T_{x} (C r i t^{o} (X) \cap X_{r e g}) =\{0\}\text{.}
\end{equation*} Now if $g \in G$ then the image of the tangent space to $G \times C r i t^{o} (X) \cap X_{r e g}$ at $g ,x$ under $d \Psi _{g ,x}$ is \begin{equation*}g \left (i L i e (K) x \oplus T_{x}(C r i t^{o} (X) \cap X_{r e g}\right )
\end{equation*} which is of real dimension $2\dim X$. \ This implies that $\Psi $ is an open mapping in the metric topology. Corollary \ref{reg} implies that $V_{d}^{c} \cap X_{r e g}$ is irreducible as a quasi-affine variety over $\mathbb{C}$ hence it is connected in the metric topology (cf. \cite{shafarevich}
Book3 Section 7.3 Theorem 7.2). We note that $C r i t^{o} (X) \cap X_{r e g} =C r i t (X) \cap X_{d}^{c} \cap X_{r e g}$. Every closed orbit intersects $C r i t (X)$ thus\begin{equation*}G C r i t^{o} (X) \cap X_{r e g} =X_{d}^{c} \cap X_{r e g}\text{.}
\end{equation*} Let $C r i t^{o} (X) \cap X_{r e g} =U_{1} \cup U_{2} \cup \ldots  \cup U_{r}$ be the decomposition into connected components. Each $U_{i}$ is $K$--invariant because $K$ is connected. Assume that $v \in G U_{i} \cap G U_{j}$. Then $v =g_{1} u_{1} =g_{2} u_{2}$ with $g_{1} ,g_{2} \in G$ and $u_{1} \in U_{i}$ and $u_{2} \in U_{j}$. Thus $g_{1}^{ -1} g_{2} u_{2} =u_{1}$. Hence Theorem \ref{Kempf-Ness} implies that there exists $k \in K$ such that $k u_{2} =u_{1}$. Thus implies $U_{i} =U_{j}$. On the other hand we have seen that the map \begin{equation*}\Psi  :G \times C r i t^{o} (X) \cap X_{r e g} \rightarrow X_{d}^{c} \cap X_{r e g} ,g ,x \longmapsto g x
\end{equation*} is an open mapping. This implies that $G U_{i}$ is open in the metric topology. Hence $r =1$. 
\end{proof}

\section{Semi-stability groups.}
In this section we consider the following situation. $V$ a finite dimensional vector space over $\mathbb{C}$ and $G \subset G L (V)$ a Zariski closed, reductive subgroup. We assume that $V^{c}$, the set of $x \in V$ such that $G x$ is closed, is not equal to $\{0\}$. If $H$ is an algebraic group then $H^{o}$ will denote its identity component. \ We may (as above) assume that there is a Hermitian
inner product, $ \langle \ldots  ,\ldots  \rangle $, on $V$ such that $G$ is invariant under adjoint. Set $U$ equal to the unitary group of $(V , \langle \ldots  ,\ldots  \rangle )$ and $K =G \cap U\text{.}$For the sake of simplicity we choose an orthonormal basis of $V$ and identify $V$ with $\mathbb{C}^{n}$ and $ \langle \ldots  ,\ldots  \rangle $ with the usual inner product on $\mathbb{C}^{n}$.

\begin{lemma}
\label{power}Let $x \in V^{c} -\{0\}$. If $z \in \mathbb{C}^{ \times } ,g \in G$ and $z g x =x$ then there exists $r =r_{x}$ depending only on $x$ such that $z^{r} =1$. Set $\tilde{G}_{x} =\{g \in G\vert g x =\xi  x$ for some $\xi  \in \mathbb{C}^{ \times }\}$ then $(\tilde{G}_{x})^{o} =(G_{x})^{o}\text{.}$ 
\end{lemma}

\begin{proof}
Assume that $x \in V^{c} -\{0\} ,z \in \mathbb{C}^{ \times } ,g \in G$ and $z g x =x$. There exists $f \in \mathcal{O} (V)^{G}$ such that $f$ is homogeneous of degree $r >0$ and $f (x) =1$. Thus \begin{equation*}1 =f (g x) =f (z^{ -1} g) =z^{ -r}\,
\end{equation*} Thus we have a morphism \begin{equation*}\chi  :\tilde{G}_{x} \rightarrow \mu _{r}
\end{equation*} ($\mu _{r}$ is the group of $r$--th roots of 1) defined by \begin{equation*}g x =\chi  (g) x\text{.}
\end{equation*} To see that $\chi $ is regular choose $\lambda  \in V^{ \ast }$ such that $\lambda  (x) =1$. If $g \in \tilde{G}_{x}$ then $\chi  (g) =\lambda  (g x)$. The lemma follows since $\chi  \left ((\tilde{G}_{x})^{o}\right ) =\{1\}$. 
\end{proof}

\begin{lemma}
If $x \in C r i t (V) -\{0\}$ then $\tilde{G}_{x}$ is invariant under $g \longmapsto g^{ \ast }$ relative to $ \langle \ldots  ,\ldots  \rangle $. In particular, $\tilde{G}_{x}$ is the Zariski closure of $K \cap \tilde{G}_{x}$. 
\end{lemma}

\begin{proof}
The Kempf-Ness theorem implies that we may assume that $x \in C r i t (V)$. Let $g \in \tilde{G}_{x}$ and $g =k e^{X}$ with $k \in G \cap U (n)$ and $X \in L i e (G)$, $X^{ \ast } =X$. then $g x =\xi  x$ and since $\xi ^{r} =1$, $\left \vert \xi \right \vert  =1$. Thus\begin{equation*}\left \Vert x\right \Vert  =\left \Vert g x\right \Vert  =\left \Vert k e^{X} x\right \Vert  =\left \Vert e^{X} x\right \Vert \text{.}
\end{equation*} Hence, the Kempf-Ness theorem implies that $X x =0.$ Thus $X \in L i e (\tilde{G}_{x})$ hence $k \in \tilde{G}_{x}$. 
\end{proof}

Obviously

\begin{lemma}
Let $x \in V^{c}$ be and let $r =r_{x}$ be as in Lemma \ref{power}. If $\xi  \in \mu _{r}$ set \begin{equation*}\tilde{G}_{x ,\xi } =\{g \in G\vert g x =\xi  x\}\text{.}
\end{equation*} then \begin{equation*}\left (\mathbb{C}^{ \times } G\right )_{x} = \cup _{\xi  \in \mu _{r}}\xi ^{ -1} \tilde{G}_{x ,\xi }\text{.}
\end{equation*} 
\end{lemma}

\begin{lemma}
$S^{1} K$ is a maximal compact subgroup of $\mathbb{C}^{ \times } G$. If $x \in C r i t (V)$ (for $G$) then the Zariski closure of $\left (S^{1} K\right )_{x}$ in $\mathbb{C}^{ \times } G$ is $\left (\mathbb{C}^{ \times } G\right )_{x}\text{.}$ 
\end{lemma}

\begin{proof}
Since $\mathbb{C}^{ \times } G$ is invariant under adjoint $\left (\mathbb{C}^{ \times } G\right ) \cap U$ is a maximal compact subgroup. Let $z \in \mathbb{C} ,k \in K$ and $X \in L i e (G)$ with $X^{ \ast } =X$. Then \begin{equation*}z k e^{X} \in U
\end{equation*} if and only if $e^{2 X} =\left \vert z\right \vert ^{ -2} I$. If $\left \vert z\right \vert  \neq 1$ then $\mathbb{C}^{ \times } I \subset G$. But we have assumed that $V^{c} \neq \{0\}$. Hence $\left \vert z\right \vert  =1$. Note that $\left (\mathbb{C}^{ \times } G\right )_{x}$ is invariant under adjoint. Indeed, if $x \neq 0$ then if $z g \in \left (\mathbb{C}^{ \times } G\right )_{x}$ with $z \in \mathbb{C}^{ \times }$ and $g \in G$ then $z^{r_{x}} =1$. Let $g =k e^{X}$ with $k \in K ,X \in L i e (G)$ and $X^{ \ast } =X$. Then\begin{equation*}\left \Vert x\right \Vert  =\left \Vert z g x\right \Vert  =\left \Vert e^{X} x\right \Vert 
\end{equation*} thus $X x =0$ hence $X \in L i e (G_{x})$. This implies that $k x =z^{ -1} x$. Now $(z g)^{ \ast } =e^{X} z^{ -1} k^{ -1}$ which fixes $x$. Thus $\left (S^{1} K\right )_{x}$ is maximal compact in $\left (\mathbb{C}^{ \times } G\right )_{x}$. 
\end{proof}

\section{A principal orbit type theorem}
We maintain the notation of the previous section. Let $X \subset V$ be an irreducible Zariski closed and $\mathbb{C}^{ \times } G$ invariant.

\begin{theorem}
Let $d (X) ,d_{1} (X)$ and $X_{d (X)}^{c}$ be as in section 2 for $G$ and assume that $d (V) =d_{1} (V)$. Then there exists an open subset, $U$, of $X_{d (X)}^{c} \cap X_{r e g}$ with complement in $X$ a union of real submanifolds of strictly lower real dimension than $2\dim X$ such that if $x ,y \in U$ then $\left (\mathbb{C}^{ \times } G\right )_{x}$ and $\left (\mathbb{C}^{ \times } G\right )_{y}$ are $G$ conjugate and if $x \in X_{d (X)}^{c} \cap X_{r e g}$ then $\left (\mathbb{C}^{ \times } G\right )_{x}$ contains a $G$--conjugate of $\left (\mathbb{C}^{ \times } G\right )_{u}$ with $u \in U$. 
\end{theorem}

\begin{proof}
Let $C r i t^{o} (X)$ be as in section 3 for $G\text{.}$ Then $S^{1} K$ acts on $C r i t^{o} (X) \cap X_{r e g}$ which we have seen is a smooth, connected, real submanifold of $V$. The principal orbit type theorem for compact Lie group actions on connected smooth manifolds (cf. \cite{Bredon}
Theorem 3.1 p. 179) implies that there exists a closed subgroup $L$ of $S^{1} K$ that is a stabilizer of a point in $C r i t^{o} (X) \cap X_{r e g}$ such that if $u \in C r i t^{o} (X) \cap X_{r e g}$ then there exists $k \in S^{1} K$ such that $k L k^{ -1} \subset \left (S^{1} K\right )_{u}$. Furthermore, the set of $u \in C r i t^{o} (X) \cap X_{r e g}$ such that $K_{u}$ is not conjugate to $L$ is a finite union of $S^{1} K$--invariant submanifolds, $Z_{i}$, of lower dimension. Let $x \in X_{d}^{c} \cap X_{r e g}$ there exists $u \in G x \cap C r i t^{o} (X) \cap X_{r e g}$. thus if $u =g_{1} x$ with $g_{1} \in G$ then $\left (\mathbb{C}^{ \times } G\right )_{u} =g_{1} \left (\mathbb{C}^{ \times } G\right )_{x} g_{1}^{ -1}$. Now $\left (S^{1} K\right )_{u}$ is $\left (\mathbb{C}^{ \times } G\right )_{u} \cap S^{1} K$ and there exists $k \in S^{1} K$ such that $k L k^{ -1} \subset \left (S^{1} K\right )_{u}$. Let $H$ be the Zariski closure of $L$ in $\mathbb{C}^{ \times } G$. Note that $L =\left (S^{1} K\right )_{w}$ with $w \in C r i t^{o} (X) \cap X_{r e g}$ so the Zariski closure of $L$ in $\mathbb{C}^{ \times } G$ is $\left (\mathbb{C}^{ \times } G\right )_{w}$. This implies that $k \left (\mathbb{C}^{ \times } G\right )_{w} k^{ -1} \subset \left (\mathbb{C}^{ \times } G\right )_{u}$ hence $g_{1}^{ -1} k \left (\mathbb{C}^{ \times } G\right )_{w} k^{ -1} g_{1} \subset \left (\mathbb{C}^{ \times } G\right )_{x}$. Take $H =\left (\mathbb{C}^{ \times } G\right )_{w}$ and $g =g_{1}^{ -1} k$. 

The set of $u \in C r i t^{o} (X) \cap X_{r e g}$ with $\left (S^{1} K\right )_{u}$ conjugate to $L$ open and dense in $C r i t^{o} (V)$. Thus since the map\begin{equation*}\Psi  :G \times C r i t^{o} (X) \cap X_{r e g} \rightarrow X_{c}^{d} \cap X_{r e g}
\end{equation*} is an open mapping in the metric topology the points in $x \in V_{d}^{c}$ such that $\left (\mathbb{C}^{ \times } G\right )_{x}$ is conjugate to $H$ is open in $V_{c}^{d}$. Note that \begin{equation*}\dim \Psi  (G \times Z_{i}\,) <2 n\text{.}
\end{equation*} To prove this we observe that the map $\eta  :\mathfrak{p} \times Z_{i} \rightarrow C r i t^{o} (X) \cap X_{r e g} ,\eta  (X ,z) =e^{x} z$ \ has image $G Z_{i}$. Thus i \begin{equation*}d \Psi _{g ,z} T_{g ,z} (G \times Z_{i}) =g d \eta _{0 ,z} (\mathfrak{p} \times Z_{i})\text{.}
\end{equation*} Observing that if $Y \in \mathfrak{p}$ and $z \in Z_{i}$ then \begin{equation*}d \eta _{0 ,z} (Y ,v) =Y z +v\text{.}
\end{equation*} and that $\dim \mathfrak{p} z =d$, we conclude that $\dim \Psi  (G \times Z_{i}\,) =d +\dim  Z_{i} <d +2 n -d$. The complement in $X_{c}^{d} \cap X_{r e g}$ of the image is a union of submanifolds of lower dimension. Since the complement of $X_{d}^{c}$ is contained the union of a finite number complex hypersurfaces the theorem follows. 
\end{proof}

Note that the
same method of proof proves

\begin{theorem}
\label{principal}Let $X$ be an irreducible Zariski closed $G$--invariant subset of $V$. Assume that $d =d (X) =d_{1} (X)\text{.}$There exists a reductive subgroup, $H$, of $G$ that is the stabilizer of a point in $X_{d}^{c} \cap X_{r e g}$ such that if $x \in X_{d}^{c} \cap X_{r e g}$ then there exists $g \in G$ such that $g H g^{ -1} \subset G_{x}$. Furthermore the set of $x \in X_{d}^{c} \cap X_{r e g}$ with stabilizer conjugate to $H$ is open in $X$ in the metric topology and its complement is a finite union of real submanifolds of real dimension strictly less than $2\dim X$. 
\end{theorem}

In Luna \cite{Luna}
Theorem 8 asserts that if $X$ smooth then there is a closed subgroup, $H$, of $G$ such that the set of $x \in X$ such that $G_{x}$ is conjugate to $H_{1}$ has Zariski interior. Thus under our hyptheses Luna's $H_{1}$ is conjugate to the $H$ in the above theorem (since a non-empty metric open dense set must intersect a Zariski open non-empty set in an irreducible
variety). Thus the metric open set in the theorem contains a non-empty Zariski open set. 

\section{Examples.}
1. In this example of a pair $(G ,V)$ the subgroup $H$ described in Theorem \ref{Principal} is a group of order $8$. However, there are elements in $V$ with trivial stabilizer. 

Let \begin{equation*}G =S L (2 ,\mathbb{C}) \otimes S L (2 ,\mathbb{C}) \otimes S L (2 ,\mathbb{C}) \otimes S L (2 ,\mathbb{C})
\end{equation*} and\begin{equation*}V =\mathbb{C}^{2} \otimes \mathbb{C}^{2} \otimes \mathbb{C}^{2} \otimes \mathbb{C}^{2}\text{.}
\end{equation*} This corresponds to the split (i.e. the corresponding real form is split over $\mathbb{R}$) symmetric pair $\left (D_{4} ,A_{1} \times A_{1} \times A_{1} \times A_{1}\right )$. $\mathcal{O} (V)^{G}$ is generated by homogeneous algebraically independent invariants of degrees $2 ,4 ,4 ,6$. Thus the nullcone is of dimension 12. This pair can be realized as follows:\begin{equation*}G =\left (S O (4 ,\mathbb{C}) \times S O (4 ,\mathbb{C})\right )/\{ \pm (I ,I)\}
\end{equation*} acting on $M_{4} (\mathbb{C})$ by $(g ,h)X =g X h^{ -1}\text{.}$ Any closed orbit must intersect the diagonal. The subgroup of $G$ that acts trivially in the diagonal is the subgroup \begin{equation*}\left \{\left (\left [\begin{array}{cccc}\varepsilon _{1} & 0 & 0 & 0 \\
0 & \varepsilon _{2} & 0 & 0 \\
0 & 0 & \varepsilon _{3} & 0 \\
0 & 0 & 0 & \varepsilon _{4}\end{array}\right ] ,\left [\begin{array}{cccc}\varepsilon _{1} & 0 & 0 & 0 \\
0 & \varepsilon _{2} & 0 & 0 \\
0 & 0 & \varepsilon _{3} & 0 \\
0 & 0 & 0 & \varepsilon _{4}\end{array}\right ]\right )\vert \varepsilon _{j} = \pm 1\right \}/\{ \pm (I ,I)\}
\end{equation*} This is a group of order $8$. The invariants mentioned above can be taken to be \begin{equation*}\mathrm{t} \mathrm{r} X X^{T} ,\det X ,\mathrm{t} \mathrm{r} \left (X X^{T}\right )^{2} ,\mathrm{t} \mathrm{r} \left (X X^{T}\right )^{3}\text{.}
\end{equation*} $G$ is a subgroup of the\ version of $D_{4}$ of adjoint type and $V$ is the $G$--invariant complement of $L i e (G)$ in $D_{4}$. Since the pair is split, $V$ contains a principal nilpotent element of $D_{4}$. Since we are in the case of adjoint type the stabilizer in $D_{4}$ of such an element, $x$, consists of unipotent elements. Thus $G_{x}$ consists of unipotent elements and is of dimension $0$ (since the orbit of $x$ is open in the null cone). Thus it is trivial. An example of such an element is ($e_{1} ,e_{2}$ the standard basis of $\mathbb{C}^{2}$)

\begin{equation*}e_{1} \otimes e_{1} \otimes e_{2} \otimes e_{2} +e_{1} \otimes e_{2} \otimes e_{1} \otimes e_{1} +e_{2} \otimes e_{1} \otimes e_{1} \otimes e_{2} +e_{2} \otimes e_{1} \otimes e_{2} \otimes e_{1}\text{.}
\end{equation*} This implies that the principal orbit type for $K =S U (2) \otimes S U (2) \otimes S U (2) \otimes S U (2)$ on $V$ has trivial stabilizer. 

Note that the group $H$ in \ref{principal} is the group of order $8$. Thus in general his principal orbit type is not a minmal orbit. 

2. This is an example with a critical
element having a trivial stabilizer in $G$ but not in $\mathbb{C}^{ \times } G$. Consider \begin{equation*}\mathbb{C}^{ \times } G =G L (2) \otimes G L (2) \otimes G L (2) \otimes G L (2) \otimes G L (2)
\end{equation*} and \begin{equation*}G =S L (2) \otimes S L (2) \otimes S L (2) \otimes S L (2) \otimes S L (2)
\end{equation*}\begin{equation*}V =\mathbb{C}^{2} \otimes \mathbb{C}^{2} \otimes \mathbb{C}^{2} \otimes \mathbb{C}^{2} \otimes \mathbb{C}^{2}\text{.}
\end{equation*} 

Let $w$ be the sum of the elements $e_{i_{1}} \otimes e_{i_{2}} \otimes e_{i_{3}} \otimes e_{i_{3}} \otimes e_{i_{3}}$ with all of the $i_{j}$ except for one, say $i_{k}$, equal to $2$ \ and $i_{k} =1$. That is \begin{equation*}w =e_{1} \otimes e_{2} \otimes e_{2} \otimes e_{2} \otimes e_{2} +e_{2} \otimes e_{1} \otimes e_{2} \otimes e_{2} \otimes e_{2} +\ldots 
\end{equation*} Then the element\begin{equation*}v =e_{1} \otimes e_{1} \otimes e_{1} \otimes e_{1} \otimes e_{1} -\frac{1}{\sqrt{3}} w
\end{equation*} is critical and $G_{v} =\{I\}$. However, the element $g \in G$ given as follows: Let $\xi ^{2} =i$ (i.e. $\xi $ is a primitive eighth root of $1$) and set \begin{equation*}A =\left [\begin{array}{cc}\xi  & \, \\
\, & \xi ^{ -1}\end{array}\right ]
\end{equation*} and \begin{equation*}g =A \otimes A \otimes A \otimes A \otimes A
\end{equation*} We have \begin{equation*}g v =\xi ^{ -3} v\text{.}
\end{equation*} In \cite{GoKrWa} we have
shown that the set of $x \in V$ such that $G x$ is closed and $\left (\mathbb{C}^{ \times } G\right )_{x}$ is trivial is open and dense in $V$. This type of element, $v$, was studied in \cite{GIT}, Theorem 5.1.4. It
also plays a role in \cite{GoKrWa} the element $g$ as above was pointed out to us by David Saurwein.

\end{document}